\documentclass[12pt]{amsart}
\usepackage{verbatim}
\usepackage[draft]{todonotes}

\usepackage{amsthm}
\usetikzlibrary{arrows}
\usepackage{blindtext}

\usepackage{amsmath}
\usepackage{stmaryrd}
\usepackage{quiver}
\usepackage{amssymb}
\usepackage{mathrsfs}
\usepackage[all]{xy}
\usepackage{tikz}
\usepackage[alphabetic,nobysame,abbrev]{amsrefs}
\usepackage{enumerate}
\RequirePackage{fix-cm}
\usepackage{bm}
\usepackage{mdframed}
\usepackage{cancel}
\usepackage{makecell}
\usepackage{diagbox}

\usepackage{amscd}
\usepackage{amsmath}
\usepackage{bm}
\usepackage{newtxtext} 
\numberwithin{equation}{section}
\usepackage{romanbar}
\usepackage[toc,page]{appendix}
\usepackage{hyperref}
\input xy
\xyoption{all}
\usepackage{color}
\RequirePackage{fix-cm}
\usepackage{bm}
\usepackage{mdframed}
\usepackage{cancel}
\usepackage{makecell}
\usepackage{diagbox}
\usepackage{pgf,tikz}
\usepackage{mathrsfs}
\usetikzlibrary{arrows}
\usetikzlibrary{arrows,decorations.markings}
\usetikzlibrary{matrix}
\usetikzlibrary{graphs}
\usetikzlibrary{backgrounds}

\definecolor{qqqqff}{rgb}{0.,0.,1.}
\definecolor{xdxdff}{rgb}{0.49019607843137253,0.49019607843137253,1.}

\definecolor{qqqqff}{rgb}{0.,0.,1.}

\usepackage{tikz}
\usepackage{graphicx}
\usepackage{hyperref}
\usepackage[utf8]{inputenc}
\usepackage{amsmath,amsfonts,amssymb,euscript,graphicx,color,tikz}

\newtheorem{lemma}{Lemma}[section]
\newtheorem{corollary}[lemma]{Corollary}
\newtheorem{theorem}[lemma]{Theorem}
\newtheorem{proposition}[lemma]{Proposition}

\theoremstyle{definition}
\newtheorem{remark}[lemma]{Remark}

\newtheorem{definition}[lemma]{Definition}

\DeclareMathAlphabet{\mathpzc}{OT1}{pzc}{m}{it} 

\DeclareMathOperator{\modd}{mod}

\DeclareMathOperator{\add}{add}
\DeclareMathOperator{\Hom}{Hom}

\DeclareMathOperator{\id}{id}
\DeclareMathOperator{\Mor}{Mor}

\DeclareMathOperator{\Ext}{Ext}

\DeclareMathOperator{\Ker}{Ker}
\DeclareMathOperator{\Cok}{Cok}

\DeclareMathOperator{\Fac}{Fac}

\newtheorem{question}[lemma]{Question}
\newtheorem*{theorem 0*}{Theorem}
\newtheorem*{theorem a*}{Theorem A}
\newtheorem*{theorem b*}{Theorem B}

\newcounter{diagram}
\numberwithin{diagram}{section}

\begin{document}
	
	\title{Model structures arising from weak cotorsion pairs}
	
	\author{Ramin Ebrahimi}
	\address{School of Mathematical Sciences, Zhejiang Normal University, Jinhua 321004, China\\ and School of Mathematics, Institute for Research in Fundamental Sciences (IPM), P.O. Box: 19395-5746, Tehran, Iran}
	\email{rebrahimi@zjnu.edu.cn / ramin.ebrahimi1369@gmail.com}
	
	\author{Rasool Hafezi}
	\address{School of Mathematics and Statistics, Nanjing University of Information Science and Technology, Nanjing, Jiangsu 210044, P.R. China}
	\email{hafezi@nuist.edu.cn}

\author{Jiaqun Wei}
	\address{School of Mathematical Science, Zhejiang Normal University, Jinhua 321004, China}
	\email{weijiaqun5479@zjnu.edu.cn}

	

	\begin{abstract}
		Let $\mathcal{A}$ be an abelian category.
		Beligiannis and Reiten proved that there is a bijective correspondence between so-called projective model structures on $\mathcal{A}$ and hereditary cotorsion pairs in $\mathcal{A}$ with a contravariantly finite core. 
		
		It is well-known that, tilting modules induce cotorsion pairs, so we may have a homotopicl interpretation of tilting modules. But a recent generalization of tilting modules, support $\tau$-tilting modules, induce weak cotorsion pairs.
		
		In this paper, we define weak projective model structures and prove that there is a bijective correspondence between weak projective model structures and left weak cotorsion pairs satisfying some mild conditions. This is a generalization of Beligiannis-Reiten correspondence from the perspective and philosophy of $\tau$-tilting theory. In particular, we prove that any support $\tau$-tilting module induce a model structure, and there is bijective correspondence between support $\tau$-tilting modules and a certain class of model structures.
	\end{abstract}
	
	\maketitle


	\section{Introduction}
It is a common situation in mathematics that, in a given category under study, we have a class of morphism (weak equivalences), which although they are not isomorphism, we would like to consider them as isomorphism. Qusi-isomorphisms in homological algebra and weak equivalences in homotopy theory are important instances.
    
There is a canonical way to do this. Indeed, given any class of morphisms $W$ in a category $\mathcal{A}$, using Gabriel-Zisman localization \cite{GZ}, one can obtain a new category $\mathcal{A}[W^{-1}]$, called the localization of $\mathcal{A}$ at $W$, which have the same object as $\mathcal{A}$, but morhisms in $W$ have become isomorphisms. Roughly speaking, we add formal inverses to all morphisms in $W$, and a morphism in $\mathcal{A}[W^{-1}]$ is a finite sequence of morphisms in $\mathcal{A}$ and the formal inverses of elements of $W$. Also, there is a canonical functor $\mathcal{A}\rightarrow\mathcal{A}[W^{-1}]$, which is universal among all functors from $\mathcal{A}$ to another category that send morphisms in $W$ to isomorphisms. 
    
Thus, $\Hom$-sets in $\mathcal{A}[W^{-1}]$ are too complicated.
Indeed, we don't even know for two objects $A,B\in\mathcal{A}$, if $\Hom_{\mathcal{A}[W^{-1}]}(A,B)$ is a set or a proper class.

Model structures were introduced by Quillen with the aim to do homotopy theory in a general category \cite{Q1,Q2,Q3}.
If the class $W$ of morphism that we would like to consider them as isomorphisms fits in a Quillen model structure $(Cof,We=W,Fib)$, then the localized category $\mathcal{A}[W^{-1}]$ is equivalent to a certain subfactor of $\mathcal{A}$ and it is much easier to be understood. 

A model structure on a category $\mathcal{A}$ is a triple $\mathcal{M}=(Cof,We,Fib)$ of classes of morphisms in $\mathcal{A}$, respectively called cofibrations, weak equivalences and fibrations satisfying some axioms, see Definition \ref{2.4}.
If our category $\mathcal{A}$ is abelian (or exact or triangulated or even extriangulated), model structures are closely related to seemingly less complicated notion of cotorsion pair.

A model structure $\mathcal{M}=(Cof,We,Fib)$ on $\mathcal{A}$ is called addmisssible if cofibrations are exactly monics with cofibrant cokernel, and fibration are exactly epics with fibrant kernel. Then, Hovey's correspondence says that there is a bijective correspondence between admissible model structures and pairs of cotorsion pairs satisfying some compatiblity conditions (Hovey twin cotorsion pairs). For more detail the reader is referred to the original work of Hovey \cite{H2} or the recent book of Gillespie \cite{G}.

On the other hand, Beligiannis and Reiten proved that some model structures can be
parametrized by just one cotorsion pair. Indeed, they defined projective model structures and proved that there is a bijective correspondence between projective model structures and hereditary cotorsion pairs with contravariantly finite core. Note that the heriditary condition is missed in \cite{BR}, and as it was proved in \cite{CLZ}, this condition is necessary. Recently, Beligiannis-Reiten correspondence has been generalized to weakly idempotent complete exact and extriangulated categories \cite{CLZ,HZZZ}.

Let $\Lambda$ be an Artin algebra. By Auslander-Reiten correspondence, any finitely generated (co)tilting $\Lambda$-module induce a cotorsion pair in $\modd\text{-}\Lambda$ \cite{AR}. So, we are able to give a homotopical interpretation of tilting modules. This was done by Beligiannis and Reiten \cite{BR}.

The following question was our motivation for the project.

\begin{question}\label{1.1}
Let $\Lambda$ be an Artin algebra. Any support $\tau$-tilting module induce a left weak cotorsion pair ($lw$-cotorsion pair for short)\cite{BZ}, see Definition \ref{3.1}.
Is there a connection between $lw$-cotorsion pairs and model structures?
If so, can we give a homotopical interpretation of support $\tau$-tilting modules?
\end{question}
We give an affirmant answer to these questions.
First, by weakening the notion of projective model structures, we define weak projective model structures, see Definition \ref{3.2}, and we prove that (see Theorem \ref{3.20}):

\begin{theorem}\label{1.2}
let $\mathcal{A}$ be an abelian category.
There is a bijective correspondence between weak projective model structures on $\mathcal{A}$, $lw$-cotorsion pairs in $\mathcal{A}$ with contravariantly finite core which satisfies some mild conditions.
\end{theorem}
Then, we use this result to give a homotopical interpretation of support $\tau$-tilting modules. Indeed, we prove that for Artin algebra $\Lambda$, there is a bijective correspondence between support $\tau$-tilting modules and model structures satisfying the following conditions (see Theorem \ref{4.5}):
   \begin{itemize}
\item[(i)]
All modules in $\modd\text{-}\Lambda$ are fibrant.
\item[(ii)]
Trivial fibrations coincides with epics with trivially fibrant kernel.
\item[(iii)]
The subcategory of trivially fibrant objects is closed under factor modules.
\end{itemize}
    
    \section*{Notations and Conventions} 
 Throughout the paper, $\mathcal{A}$ is mostly an abelian category, but in some places $\mathcal{A}$ may be an arbitrary category. All subcategories are full, so a subcategory is identified with its class of objects.

For a subcategory $\mathcal{S}$ of $\mathcal{A}$, the left and right $\Hom$-orthogonal and $\Ext^1$-orthogonal subcategories of $\mathcal{S}$ are defined as

\begin{align*}
{}^{\perp_1}\mathcal{S}:=\{X\in \mathcal{A}\mid\; \Ext^1_{\mathcal{A}}(X,S)=0,\; \forall S\in \mathcal{S}\},\\
\mathcal{S}^{\perp_1}:=\{X\in \mathcal{A}\mid\; \Ext^1_{\mathcal{A}}(S,X)=0,\; \forall S\in \mathcal{S}\}.
\end{align*}
 
Also, for two subcategories $\mathcal{X}$ and $\mathcal{Y}$ of $\mathcal{A}$ and a non-negative integer $i$, $\Ext_{\mathcal{A}}^i(\mathcal{X},\mathcal{Y})=0$ means that $\Ext_{\mathcal{A}}^i(X,Y)=0$ for all $X\in\mathcal{X}$ and all $Y\in\mathcal{Y}$.

\section{Preliminaries on model structures and cotorsion pairs}
In this section we will recall the definitions and basic properties of model structures and cotorsion pairs. Also we recall the correspondence of Beligiannis and Reiten.
For more details on general model structure the reader is refered to Quillen original work \cite{Q1} and also Hovey's book \cite{H1}. For model structures on abelian categories we recommend the recent book of Gillespie \cite{G}.

Let start from cotorsion pairs.
\begin{definition}\label{2.1}
Let $\mathcal{A}$ be an abelian category.
A pair $(\mathcal{X},\mathcal{Y})$ of subcategories of $\mathcal{A}$ is called a {\it (complete) cotorsion pair} if
    \begin{itemize}
    \item[(1)]
    $\mathcal{X}$ and $\mathcal{Y}$ are closed under direct summands.
    \item[(2)]
    $\Ext_{\mathcal{A}}^1(\mathcal{X},\mathcal{Y})=0$.
    \item[(3)]
    For each $A\in \mathcal{A}$, there exist short exact sequences
    \begin{align}
    0\rightarrow A\overset{g^A}{\longrightarrow} Y^A\longrightarrow X^A\rightarrow 0\label{eq2.1}\\
    0\rightarrow Y_A\longrightarrow X_A\overset{f_A}{\longrightarrow} A\rightarrow\label{eq2.2} 0
    \end{align}
    where $X_A,X^A\in \mathcal{X}$ and $Y_A,Y^A\in \mathcal{Y}$.
    \end{itemize}
    The intersection $\omega:=\mathcal{X}\cap\mathcal{Y}$ is called the {\it core} of the cotorsion pair.
    $(\mathcal{X},\mathcal{Y})$ is called a {\it hreditary} cotorsion pair if $\Ext_{\mathcal{A}}^i(\mathcal{X},\mathcal{Y})=0$ for every $i\geq 1$.
\end{definition}
The following lemma gives some necessary and sufficient conditions for a cotorsion pair to be hereditary.
We use this lemma as motivation to define hereditary left weak cotorsion pairs, see Definition \ref{3.6}.

\begin{lemma}$($\cite[6.17]{S}$)$\label{2.2}
The following are equivalent for a cotorsion pair $(\mathcal{X},\mathcal{Y})$.
\begin{itemize}
\item[(1)]
$(\mathcal{X},\mathcal{Y})$ is a hereditary cotorsion pair.
\item[(2)]
$\mathcal{X}$ is closed under kernel of epimorphisms.
\item[(3)]
$\mathcal{Y}$ is closed under cokernel of monomorphisms.
\end{itemize}
\end{lemma}
The short exact sequences \eqref{eq2.1} and \eqref{eq2.2} are called {\it approximation sequences}, because of the following lemma. First let us recall the concept of approximation, which is fundamental both for model structures and cotorsion pairs. Let $\mathcal{B}$ be a subcategory of $\mathcal{A}$, and $A\in\mathcal{A}$. A {\it right $\mathcal{B}$-approximation} for $A$ is a morphism $f_A:B\rightarrow A$ for some object $B\in \mathcal{B}$ such that any other map $f':B'\rightarrow A$ with $B'\in \mathcal{B}$ factor through $f$. If any object of $\mathcal{A}$ have a right $\mathcal{B}$-approximation, we say that $\mathcal{B}$ is a {\it contravariantly finite} subcategory of $\mathcal{A}$.
The concepts of {\it left $\mathcal{B}$-approximation} and {\it covariantly finite} subcategory are defined dually.
    
\begin{lemma}\label{2.3}
Keeping the notations of Definition \ref{2.1} we have
    \begin{itemize}
    \item[(1)]
    $g^A$ is a left $\mathcal{Y}$-approximation for $A$.
    \item[(2)]
    $f_A$ is a right $\mathcal{X}$-approximation for $A$
    \item[(3)]
    $\mathcal{X}={}^{\perp_1}\mathcal{Y}$.
    \item[(4)]
    $\mathcal{Y}=\mathcal{X}^{\perp_1}$.
    \end{itemize}
     \begin{proof}
     For the reader's convenience, let sketch the proof of $(2)$. Given a morphism $t:X'\rightarrow A$ with $X'\in\mathcal{X}$, we can construct the following pull back diagram.
     \[\begin{tikzcd}
	0 & {Y_A} & E & {X'} & 0 \\
	0 & {Y_A} & {X_A} & A & 0
	\arrow[from=1-1, to=1-2]
	\arrow[from=1-2, to=1-3]
	\arrow["id"', from=1-2, to=2-2]
	\arrow[from=1-3, to=1-4]
	\arrow[from=1-3, to=2-3]
	\arrow[from=1-4, to=1-5]
	\arrow["{f'}", from=1-4, to=2-4]
	\arrow[from=2-1, to=2-2]
	\arrow[from=2-2, to=2-3]
	\arrow["{f_A}", from=2-3, to=2-4]
	\arrow[from=2-4, to=2-5]
\end{tikzcd}\]

   Then $\Ext_{\mathcal{A}}^1(\mathcal{X},\mathcal{Y})=0$ implies that the top row is an split short exact sequence. Thus $f'$ factors through $f_A$.
     \end{proof}
    \end{lemma}
    
In this paper we are only concern with complete cotorsion pairs, and we just call them cotorsion pairs.
	
Although we are interested in abelian categories, let start with the definition of model structure on a general category. Recall that, by a retract of a morphism $f$ we mean a morphism $g$ for which there is a commutative diagram 
	\[\begin{tikzcd}
	X & A & X \\
	Y & B & Y
	\arrow[from=1-1, to=1-2]
	\arrow["{id_X}"{description}, curve={height=-12pt}, from=1-1, to=1-3]
	\arrow["g", from=1-1, to=2-1]
	\arrow[from=1-2, to=1-3]
	\arrow["f", from=1-2, to=2-2]
	\arrow["g", from=1-3, to=2-3]
	\arrow[from=2-1, to=2-2]
	\arrow["{id_Y}"{description}, curve={height=12pt}, from=2-1, to=2-3]
	\arrow[from=2-2, to=2-3]
\end{tikzcd}\]

This means that in the morphism category of $\mathcal{A}$ (where the objects are morphisms of $\mathcal{A}$ and morphisms are commutative squares) $g$ is a retract of $f$ in the regular sense.

Let $i$ and $p$ be two morphisms in $\mathcal{A}$. We say that $i$ has {\it left lifting property} with respect to $p$, or equivalently $p$ has {\it right lifting property} with respect to $i$, if for every commutative square as follows the dashed morphism exists, making both triangles commutative.
\[\begin{tikzcd}
	A & C \\
	B & D
	\arrow[from=1-1, to=1-2]
	\arrow["i"', from=1-1, to=2-1]
	\arrow["p", from=1-2, to=2-2]
	\arrow[dashed, from=2-1, to=1-2]
	\arrow[from=2-1, to=2-2]
\end{tikzcd}\]
For a class $\mathcal{S}$ of morphisms we denote by ${}^{\boxslash}\mathcal{S}$
(resp. $\mathcal{S}{}^{\boxslash}$) the class of all morphisms that have left lifting property (resp. right lifting property) with respect to all morphisms in $\mathcal{S}$.

\begin{definition}$($\cite{Q1,H1,G}$)$\label{2.4}
A {\it model structure} on a category $\mathcal{A}$ is a triple $\mathcal{M}=(Cof,We,Fib)$ of classes of morphisms in $\mathcal{A}$, respectively called {\it cofibrations}, {\it weak equivalences}, and {\it fibrations}, satisfying the following axioms. By definition, we call a map a {\it trivial cofibration} if it is both a cofibration and a weak equivalence. Similarly a {\it trivial fibration} is both a fibration and a weak equivalence. Trivial cofibrations and trivial fibrations are denoted by $TCof$ and $TFib$ respectively. 
\begin{itemize}
\item[(M1)]
(Lifting Axiom) Morphisms in $TCof$ have left lifting property with respect to morphism in $Fib$, and morphisms in $Cof$ have left lifting property with respect to morphisms in $TFib$.
\item[(M2)]
(Factorization Axiom) Any morphism $f$ may be factored as $f = pi$ in two ways. First, where $p$ is a fibration and $i$ is a trivial cofibration. Second, where $p$ is a trivial fibration and $i$ is a cofibration. 
\item[(M3)]
(2 out of 3 Axiom) Let $gf$ be a composition of two morphisms in $\mathcal{A}$. If two out of three of the morphisms $f$, $g$, and $gf$ are a weak equivalence then so is the third. 
\item[(M4)]
(Retracts Axiom) The class of fibrations, cofibrations, and weak equivalences are each closed under retracts. 
\end{itemize}
\end{definition}
A {\it model category} is a category $\mathcal{A}$ endowed with a model structure $\mathcal{M}=(Cof,We,Fib)$ which satisfies the following axiom.
\begin{itemize}
\item[(M0)]
All finite limits and finite colimits exist in $\mathcal{A}$.
\end{itemize}
\begin{remark}\label{2.5}
Let $\mathcal{L}$ and $\mathcal{R}$ be two classes of morphisms closed under retracts in a general category $\mathcal{A}$. The pair $(\mathcal{L},\mathcal{R})$ is called a {\it weak factorization system} if morphisms in $\mathcal{L}$ have left lifting property with respect to morphisms in $\mathcal{R}$ and any morphism $f$ factors as $f=pi$, where $i\in \mathcal{L}$ and $p\in\mathcal{R}$. So, in any model structure, $(TCof,Fib)$ and $(Cof,TFib)$ are both weak factorization systems.
\end{remark}

The following two lemmas are some consequences of the axioms of model structures that we will use frequently.
	
\begin{lemma}$($\cite{H1}$)$\label{2.6}
Let $\mathcal{M}=(Cof,We,Fib)$ be a model structure on the category $\mathcal{A}$.
	\begin{itemize}
	\item[(1)]
	$Cof={}^{\boxslash}TFib$ and $TFib=Cof{}^{\boxslash}$.
	\item[(2)]
	$TCof={}^{\boxslash}Fib$ and $Fib=TCof{}^{\boxslash}$.
	\item[(3)]
	Weak equivalences are exactly the morphisms that can be factored as a trivial cofibration followed by a trivial fibration, i.e. $We=TFib \circ TCof$.
	\end{itemize}
\end{lemma}

\begin{lemma}\label{2.7}
Let $\mathcal{M}=(Cof,We,Fib)$ be a model structure on the category $\mathcal{A}$.
		\begin{itemize}
		\item[(1)]
		All the classes $Cof$, $TCof$, $Fib$, $TFib$ and $We$ are closed under composition and contain isomorphisms.
		\item[(2)]
		Morphisms in $Cof$ and $TCof$ are closed under push out, meaning that for a given push out square
		\[\begin{tikzcd}
	    \bullet & \bullet \\
	    \bullet & \bullet
	    \arrow[from=1-1, to=1-2]
	    \arrow["f"', from=1-1, to=2-1]
	    \arrow["{f'}", from=1-2, to=2-2]
	    \arrow[from=2-1, to=2-2]
        \end{tikzcd}\]
        if $f$ is a (trivial) cofibration, then so is $f'$.
		\item[(3)]
		Dually, morphisms in $Fib$ and $TFib$ are closed under pull back.
		\end{itemize}
	\begin{proof}
	We prove that for any class of morphisms $\mathcal{S}$, ${}^{\boxslash}\mathcal{S}$ is closed under push out. This together with Lemma \ref{2.6} proves $(2)$. We will use this argument in the the future. $(3)$ is dual of $(2)$ and $(1)$ is easy to proof.
	
	Let $f\in{}^{\boxslash}\mathcal{S}$ and $f'$ be a push out of $f$. We want to prove that $f'\in{}^{\boxslash}\mathcal{S}$. Given the commutative diagram
	\[\begin{tikzcd}
	A & C & X \\
	B & D & Y
	\arrow["a", from=1-1, to=1-2]
	\arrow["f"', from=1-1, to=2-1]
	\arrow["r", from=1-2, to=1-3]
	\arrow["{f'}"', from=1-2, to=2-2]
	\arrow["s", from=1-3, to=2-3]
	\arrow["b"', from=2-1, to=2-2]
	\arrow["t"', from=2-2, to=2-3]
\end{tikzcd}\]
with $s\in\mathcal{S}$, because $f\in {}^{\boxslash}\mathcal{S}$, there is a morphism $\lambda_1:B\rightarrow X$ such that $\lambda_1 f=ra$ and $tb=s\lambda_1$.
Then, because the left-hand square is push out, there is a morphism $\lambda:D\rightarrow X$ such that $\lambda f'=r$ and $\lambda b=\lambda_1$. Finally, using the universal property of push out we can prove that $s\lambda=t$. So, $\lambda$ is a lift for the right-hand square.
\end{proof}
\end{lemma}
	
A model structure is the data of three classes of morphism $Cof$, $Fib$ and $We$ (and two other related classes $TCof$ and $TFib$). When $\mathcal{A}$ is an abelian category, we can determine “good" model structures by some classes of objects, and this makes the understanding of the model structure and its homotopy category much simpler.
So let start with the following definition, which is the first step toward this goal.
	
\begin{definition}$($\cite[Page 88]{G}$)$\label{2.8}
Let $\mathcal{M}=(Cof,We,Fib)$ be a model structure on the category $\mathcal{A}$ and $A\in\mathcal{A}$.
\begin{itemize}
\item[(1)]
We say $A$ is {\it trivial} if $0\rightarrow A$ is a weak equivalence.
By the $2$ out of $3$ Axiom, we see that this is equivalent to saying that $A\rightarrow 0$ is a weak
equivalence. The class of trivial objects is denoted by $\mathpzc{W}$.
\item[(2)] 
We say $A$ is {\it cofibrant} if $0\rightarrow A$ is a cofibration. The class of cofibrant objects is denoted by $\mathpzc{C}$.
\item[(3)]
We say $A$ is {\it fibrant} if $A\rightarrow 0$ is a fibration. The class of fibrant objects is denoted by $\mathpzc{F}$.
\item[(4)]
We say $A$ is {\it trivially cofibrant} if it is both trivial and cofibrant. In other
words, $0\rightarrow A$ is a trivial cofibration. The class of trivially cofibrant objects is denoted by $\mathpzc{TC}$. 
\item[(5)]
We say $A$ is {\it trivially fibrant} if it is both trivial and cofibrant. In other
words, $A\rightarrow 0$ is a trivial fibration. The class of trivially fibrant objects is denoted by $\mathpzc{TF}$. 
\end{itemize}
\end{definition}

The following lemma is a direct consequence of Lemma \ref{2.7}, because kernel is pull back along the zero morphism, and cokernel is push out along the zero morphism.

\begin{lemma}\label{2.9}
\begin{itemize}
\item[(1)]
The cokernel of any (trivial) cofibration is a (trivially) cofibrant object.
\item[(2)]
The kernel of any (trivial) fibration is a (trivially) fibrant object.
\end{itemize}
\end{lemma}
	
It turns out that the pairs $(\mathpzc{C},\mathpzc{TF})$ and  $(\mathpzc{TC},\mathpzc{F})$ have some properties similar to cotorsion pairs. Indeed, if we consider the factorizations for the morphisms $0\rightarrow A$ and $A\rightarrow 0$
\[\begin{tikzcd}
& {X_A} &&& {Y^A} & \\
0 && A & A && 0 \\
& {C_A} &&& {F^A}
\arrow["{\phi_A}", from=1-2, to=2-3]
\arrow["p", from=1-5, to=2-6]
\arrow["i", from=2-1, to=1-2]
\arrow[from=2-1, to=2-3]
\arrow["j"', from=2-1, to=3-2]
\arrow["{\psi^A}", from=2-4, to=1-5]
\arrow[from=2-4, to=2-6]
\arrow["{g^A}"', from=2-4, to=3-5]
\arrow["{f_A}"', from=3-2, to=2-3]
\arrow["q"', from=3-5, to=2-6]
\end{tikzcd}\]
with respect to weak factorization systems $(Cof,TFib)$ and $(TCof,Fib)$ (i.e. $i, \psi^A\in TCof$, $\phi_A, p\in Fib$, $j, g^A\in Cof$ and $f_A, q\in TFib$), we have:

	\begin{proposition}$($\cite[VII, Proposition 2.1]{BR}$)$\label{2.10}
	Let $\mathcal{M}=(Cof,We,Fib)$ be a model structure on $\mathcal{A}$.
	\begin{itemize}
	\item[(1)]
	For any object $A\in \mathcal{A}$ there exists a short exact sequence
	\begin{equation*}
	0\rightarrow Y_A\rightarrow C_A\overset{f_A}{\longrightarrow}A
	\end{equation*}
	where $f_A$ is a trivial fibration and a right $\mathpzc{C}$-approximation, and $Y_A\in\mathpzc{TF}$.
	In particular, $\mathpzc{C}$ is a contravariantly finite subcategory.
	\item[(2)]
	For any object $A\in \mathcal{A}$ there exists a short exact sequence
	\begin{equation*}
	0\rightarrow F_A\rightarrow X_A\overset{\phi_A}{\longrightarrow}A
	\end{equation*}
	where $\phi_A$ is a fibration and a right $\mathpzc{TC}$-approximation and $F_A\in\mathpzc{F}$.
	In particular, $\mathpzc{TC}$ is a contravariantly finite subcategory.
	\item[(3)]
	For any object $A\in \mathcal{A}$ there exists a short exact sequence
	\begin{equation*}
	A\overset{g^A}{\longrightarrow} F^A\rightarrow X^A\rightarrow 0
	\end{equation*}
	where $g^A$ is a cofibration and a left $\mathpzc{TF}$-approximation and $X^A\in\mathpzc{C}$.
	In particular, $\mathpzc{TF}$ is a covariantly finite subcategory.
	\item[(4)]
	For any object $A\in \mathcal{A}$ there exists a short exact sequence
	\begin{equation*}
	A\overset{\psi^A}{\longrightarrow} Y^A\rightarrow C^A\rightarrow 0
	\end{equation*}
	where $\psi^A$ is a trivial cofibration and a left $\mathpzc{F}$-approximation and $C^A\in\mathpzc{TC}$.
	In particular, $\mathpzc{TF}$ is a covariantly finite subcategory.
	\end{itemize}
	\end{proposition}
	
	The above proposition together with the following proposition, make the relation between cotorsion pairs and model structures clear.
	
	\begin{proposition}$($\cite[VIII, Lemma 3.2]{BR}$)$\label{2.11}
	\begin{itemize}
	\item[(1)]
	If any trivial fibration is epic, then: $^{\perp_1}\mathpzc{TF}\subseteq\mathpzc{C}$.
	\item[(2)]
	If any fibration is epic, then: $^{\perp_1}\mathpzc{F}\subseteq \mathpzc{TC}$.
	\item[(3)]
	If any trivial cofibration is monic, then: $\mathpzc{TC}^{\perp_1}\subseteq\mathpzc{F}$.
	\item[(4)]
	If any cofibration is monic, then: $\mathpzc{C}^{\perp_1}\subseteq\mathpzc{TF}$.
	\item[(5)]
	If any monic with cokernel in $\mathpzc{TC}$ is a trivial cofibration, then: $\Ext_{\mathcal{A}}^1(\mathpzc{TC},\mathpzc{F})=0$.
	\item[(6)]
	If any epic with kernel in $\mathpzc{TF}$ is a trivial fibration, then: $\Ext_{\mathcal{A}}^1(\mathpzc{C},\mathpzc{TF})=0$.
	\end{itemize}
	\end{proposition}

   The following lemma, known as Lifting-Extension Lemma, relating vanishing of $\Ext$ and lifting property is crucial.
	
	\begin{lemma}$($\cite[VII, Lemma 3.1]{BR}$)$\label{2.12}
	For two object $X$ and $Y$ in the abelian category $\mathcal{A}$ the following are equivalent.
	\begin{itemize}
	\item[(1)]
	$\Ext_{\mathcal{A}}^1(X,Y)=0$.
	\item[(2)]
	For any commutative diagram with exact columns
	\[\begin{tikzcd}
	& 0 \\
	0 & Y \\
	A & C \\
	B & D \\
	X & 0 \\
	0
	\arrow[from=1-2, to=2-2]
	\arrow[from=2-1, to=3-1]
	\arrow[from=2-2, to=3-2]
	\arrow[from=3-1, to=3-2]
	\arrow[from=3-1, to=4-1]
	\arrow[from=3-2, to=4-2]
	\arrow[dashed, from=4-1, to=3-2]
	\arrow[from=4-1, to=4-2]
	\arrow[from=4-1, to=5-1]
	\arrow[from=4-2, to=5-2]
	\arrow[from=5-1, to=6-1]
\end{tikzcd}\]
the dashed arrow exists, making both triangles commutative.
	\end{itemize}
	\end{lemma}
	
	Now, we can state the following theorem from \cite{BR,CLZ}, which makes the relation between model structures and cotorsion pairs precise.
	
	\begin{theorem}\label{2.13}
	Let $\mathcal{M}=(Cof,We,Fib)$ be a model structure on $\mathcal{A}$.
	\begin{itemize}
	\item[(1)] The following are equivalent.
          \begin{itemize}
          \item[(i)]
          $(\mathpzc{C},\mathpzc{TF})$ is a cotorsion pair.
          \item[(ii)]
          Any cofibration is monic, and
	   \[TFib=\{\text{epics with kernel in } \mathpzc{TF}\}.\]
          \end{itemize}
	\item[(2)] The following are equivalent.
       \begin{itemize}
       \item[(i)]
       $(\mathpzc{TC},\mathpzc{F})$ is a cotorsion pair.
       \item[(ii)]
       Any fibration is epic, and
	   \[TCof=\{\text{monics with cokernel in } \mathpzc{C}\}.\]
       \end{itemize}
	\end{itemize}
    \begin{proof}
    Let sketch the proof of $(1)$.
    First note that because cofibrations and trivial fibrations are closed under retracts, $\mathpzc{C}$ and $\mathpzc{TF}$ are closed under direct summands.
  
  $(ii\Rightarrow i):$ Given a short exact sequence $0\rightarrow Y\rightarrow E\overset{f}{\rightarrow} X\rightarrow 0$, with $X\in\mathpzc{C}$ and $Y\in\mathpzc{TF}$, by $(ii)$ $f$ is a trivial fibration. So $0\rightarrow X$ has left lifting property with respect to $f$. This means that the short exact sequence splits, proving $\Ext_{\mathcal{A}}^1(\mathpzc{C},\mathpzc{TF})=0$.
The existence of approximation sequences follows from Proposition \ref{2.10} and $(ii)$.
    
$(ii\Rightarrow i):$Now assume that $(\mathpzc{C},\mathpzc{TF})$ is a cotorsion pair, and $p:M\rightarrow N$ be a trivial fibration. Take $f:C\rightarrow N$ to be a right $\mathpzc{C}$-approximation for $N$, which is epic. Then the lift for the following square proves that $p$ is also epic.
    \[\begin{tikzcd}
	0 & M \\
	C & N
	\arrow[from=1-1, to=1-2]
	\arrow[from=1-1, to=2-1]
	\arrow["p", from=1-2, to=2-2]
	\arrow[dashed, from=2-1, to=1-2]
	\arrow["f", from=2-1, to=2-2]
\end{tikzcd}\]

 Similarly, one can prove that cofibrations are monomorphism. It remains to prove that if $p:M\rightarrow N$ is an epimorphism with  $\ker(p)\in\mathpzc{TF}$, then $p$ is a trivial fibration. For that purpose, by Lemma \ref{2.6}, we need to prove that $p$ has right lifting property with respect to all cofibrations. But this follows easily from Lifting-Extension Lemma.
    \end{proof}
	\end{theorem}
	
Motivated by the above theorem, one may want to define an special type of model structures that can be parametrized by cotorsion pairs.
Not that, these model structures were called projective model structure in \cite{BR}, and were called weakly projective model structure in \cite{CLZ}. We use the former terminology because in the next section we will define weak projective model structure by weakening this definition, see Definition \ref{3.2}.
	
	\begin{definition}\label{2.14}
	Let $\mathcal{M}=(Cof,We,Fib)$ be a model structure on $\mathcal{A}$.
	$\mathcal{M}$ is called a {\it projective} model structure, if
	     \begin{itemize}
	     \item[(1)]
	     Cofibrations are monic.
	     \item[(2)]
	     $TFib=\{\text{epics with kernel in }\mathpzc{TF}\}$.
	     \item[(3)]
	     Any object is fibrant, i.e. $\mathpzc{F}=\mathcal{A}$.
	     \end{itemize}
	\end{definition}

\begin{remark}\label{2.15}
Let $\mathcal{M}=(Cof,We,Fib)$ be a projective model structure on abelian category $\mathcal{A}$. By definition, any cofibration is monic, and by Lemma \ref{2.9} the cokernel is in $\mathpzc{C}$.
Conversely, by Lifting-Extension Lemma, any monic whith a cofibrant cokernel has left lifting property with respect to all trivial fibrations. So, it is indeed a cofibration by Lemma \ref{2.6}. Thus, $Cof=\{\text{monics with cokernel in }\mathpzc{C}\}$.
\end{remark}

Now we can state the correspondence of Beligiannis and Reiten.
This correspondence has been generalized to weakly idempotent complete exact and extriangulated categories \cite{CLZ,HZZZ}.

\begin{theorem}\label{2.16}
Let $\mathcal{A}$ be an abelian category.
There is a bijective correspondence between projective model structures on $\mathcal{A}$ an hereditary cotorsion pairs in $\mathcal{A}$ with a contravariantly finite core.
Under this bijection, the projective model structure $(Cof,We,Fib)$ is sent to the cotorsion pair $(\mathpzc{C},\mathpzc{TF})$.
\end{theorem}

In the next section, we prove a generalization of Theorem \ref{2.16}. And as a corollary, we can recover this theorem. See Corollary \ref{3.21}.

	
\section{Weak projective model structures and left weak cotorsion pairs}
In this section we revise Beligiannis-Reiten correspondence, stated in Theorem \ref{2.16}.
Let start from some motivations.
From Proposition \ref{2.10} we know that approximation sequences obtained from a given model structure fail to be short exact sequences in general. This remind us of the notion of weak cotorsion pairs \cite{BZ}, a notion originated in $\tau$-tilting theory \cite{AIR}.
So, one may wonder, if in Beligiannis-Reiten correspondence, we constitute cotorsion pairs with weak cotorsion pairs, what kind of model structures we get?
We will give an affirmant answer to the question in this section. Let start from the definition of left weak cotorsion pair.

\begin{definition}$($\cite{BZ}$)$\label{3.1}
A pair $(\mathcal{X},\mathcal{Y})$ of subcategories of $\mathcal{A}$ is called a {\it left weak cotorsion pair} ({\it $lw$-cotorsion pair} for short) if
    \begin{itemize}
    \item[(1)]
    $\mathcal{X}$ and $\mathcal{Y}$ are closed under direct summands. 
    \item[(2)]
    $\Ext^1(\mathcal{X},\mathcal{Y})=0$.
    \item[(3)]
    For each $A\in \mathcal{A}$, there exist short exact sequences
    \begin{align}
     A\overset{g^A}{\longrightarrow} Y^A\longrightarrow X^A\rightarrow 0\label{eq3.1}\\
    0\rightarrow Y_A\longrightarrow X_A\overset{f_A}{\longrightarrow} A\rightarrow\label{eq3.2} 0
    \end{align}
    where $X_A,X^A\in \mathcal{X}$ and $Y_A,Y^A\in \mathcal{Y}$. And $g^A$ is a left $\mathcal{Y}$-approximation for $A$. 
\end{itemize}
The notion of {\it right weak cotorsion pair} is defined dually. 
\end{definition}

Motivated by the question in the beginning of this section we are interested in the model structures given by the following definition. The reader is also recommended to compare it with Definition \ref{2.14}.
\begin{definition}\label{3.2}
	Let $\mathcal{M}=(Cof,We,Fib)$ be a model structure on abelian category $\mathcal{A}$.
	$\mathcal{M}$ is called a {\it weak projective model structure}, if
	     \begin{itemize}
	     \item[(1)]
	     $TFib=\{\text{epics with kernel in }\mathpzc{TF}\}$.
	     \item[(2)]
	     Any object is fibrant, i.e. $\mathpzc{F}=\mathcal{A}$.
	     \end{itemize}
	\end{definition}

\begin{remark}\label{3.3}
Note that a projective model structures are exactly weak projective model structure such that cofibrations are all monic.
\end{remark}

\subsection{From weak projective model structures to left weak cotorsion pairs}
In this subsection, we show that any weak projective model structure induces a $lw$-cotorsion pair, and examine the characteristics of these $lw$-cotorsion pairs.

\begin{proposition}\label{3.4}
   Let $\mathcal{M}=(Cof,We,Fib)$ be a weak projective model structure on $\mathcal{A}$. Then $(\mathpzc{C},\mathpzc{TF})$ is a $lw$-cotorsion pair.
   \begin{proof}
   From Proposition \ref{2.11} we have that $\Ext_{\mathcal{A}}^1(\mathpzc{C},\mathpzc{TF})=0$ and the existence of approximation sequences follows from Proposition \ref{2.10} and the assumption that trivial fibrations are epic.
   \end{proof}
   \end{proposition}  
	
	\begin{proposition}\label{3.5}
   Let $\mathcal{M}=(Cof,We,Fib)$ be a weak projective model structure on $\mathcal{A}$. Then the $lw$-cotorsion pair $(\mathpzc{C},\mathpzc{TF})$ satisfies the following conditions.
   \begin{itemize}
   \item[(1)]
   $\mathpzc{TF}$ is closed under extensions, and cokernel of monics.
   \item[(2)]
   The core is contravariantly finite.
   \item[(3)]
   For each $A\in\mathcal{A}$, the approximation sequence \eqref{eq3.1} can be chosen in such a way that $g^A$ has left lifting property with respect to all epicss with kernel in $\mathpzc{TF}$ (=trivial fibrations).
   \end{itemize}
   \begin{proof}
   Let $0\rightarrow Y_1\rightarrow E\rightarrow Y_2\rightarrow 0$ be a short exact sequence with $Y_1,Y_2\in \mathpzc{TF}$. By the definition of weak projective model structures, $E\rightarrow Y_2$ is a trivial fibration and in particular a weak equivalence. Because $Y_2\rightarrow 0$ is also a weak equivalence by assumption, $E\rightarrow 0$ is also a weak equivalence by $2$ out of $3$ axiom. This prove that $E$ is a trivial object, and since all objects are fibrant we conclude $E\in \mathpzc{TF}$. So, $\mathpzc{TF}$ is closed under extensions. Now let $0\rightarrow Y\rightarrow Y'\rightarrow C\rightarrow 0$ be a short exact sequence with $Y,Y'\in \mathpzc{TF}$. By the definition of weak projective model structures, $Y'\rightarrow C$ is a trivial fibration and in particular a weak equivalence. Again by using $2$ out of $3$ axiom we can prove that $C\in \mathpzc{TF}$. These prove $(1)$.
  
  Because all objects are fibrant, we have $\omega=\mathpzc{C}\cap\mathpzc{TF}=\mathpzc{C}\cap\mathpzc{W}\cap\mathpzc{F}=\mathpzc{C}\cap\mathpzc{W}=\mathpzc{TC}$,
  which is contravariantly finite by Proposition \ref{2.10}. This proves $(2)$.
  
  $(3)$ Follows immediately from Proposition \ref{2.10}.
   \end{proof}
   \end{proposition}
Therefore, for a $lw$-cotorsion pair $(\mathcal{X},\mathcal{Y})$ to be able to induce a model structure, it must at least meets the conditions mentioned in the above proposition. Motivated by this, we give the following definition.

\begin{definition}\label{3.6}
Let $(\mathcal{X},\mathcal{Y})$ be a $lw$-cotorsion pair.
\begin{itemize}
\item[(1)]
We say that $(\mathcal{X},\mathcal{Y})$ is a {\it hereditary $lw$-cotorsion pair}, if $\mathcal{Y}$ is closed under extensions and cokernel of monics.
\item[(2)]
We say that $(\mathcal{X},\mathcal{Y})$ is a {\it homotopicaly $lw$-cotorsion pair}, if for each $A\in\mathcal{A}$, the approximation sequence \eqref{eq3.1} can be chosen in such a way that $g^A$ has left lifting property with respect to all epics with kernel in $\mathcal{Y}$.
\item[(3)]
We say that $(\mathcal{X},\mathcal{Y})$ is a {\it homotopicaly hereditary $lw$-cotorsion pair}, or $hh$-$lw$-cotorsion pair for short, If $(\mathcal{X},\mathcal{Y})$ satisfies both conditions $(1)$ and $(2)$.
\end{itemize}
\end{definition}

\begin{remark}\label{3.7}
\begin{itemize}
\item[(1)]
From the approximation sequence \eqref{eq3.2} in the definition of $lw$-cotorsion pair, we can easily see that $\mathcal{X}={}^{\perp_1}\mathcal{Y}$, which implies that $\mathcal{X}$ is closed under extensions. But $\mathcal{Y}$ may not be closed under extensions.
However, both are automatic for cotorsion pairs, see Lemma \ref{2.3}.
\item[(2)]
Any cotorsion pair $(\mathcal{X},\mathcal{Y})$ satisfies the condition of being homotopicaly. This follows from Lifting-Extension Lemma.
But for $lw$-cotorsion pairs, being homotopicaly is not guaranteed.
Actually it is similar to the existence of approximation sequences, but it is stronger.
Assuming that in the exact sequence \eqref{eq3.1} $g^A$ has left lifting property with respect to all epics with kernel in $\mathcal{Y}$, for any $Y\in\mathcal{Y}$ and any $f:A\rightarrow Y$, the following commutative square admits a lift as suggested by the dashed arrow.
\[\begin{tikzcd}
A & Y \\
{Y^A} & 0
\arrow["f", from=1-1, to=1-2]
\arrow["{g^A}"', from=1-1, to=2-1]
\arrow[from=1-2, to=2-2]
\arrow[dashed, from=2-1, to=1-2]
\arrow[from=2-1, to=2-2]
\end{tikzcd}\]
This means that $g^A$ is a left $\mathcal{Y}$-approximation for $A$.
\item[(3)]
As we will see in Proposition \ref{4.4}, $lw$-cotorsion pairs arising from $\tau$-tilting theory satisfy both conditions of being homotopicaly and hereditary.
\end{itemize}
\end{remark}

\subsection{From $hh$-$lw$-cotorsion pairs to weak projective model structures}
In this section, for a given $hh$-$lw$-cotorsion pairs $(\mathcal{X},\mathcal{Y})$ with a contravariantly finite core, we construct a weak projective model structure.

Fix a hereditary $hh$-$lw$-cotorsion pairs $(\mathcal{X},\mathcal{Y})$, and assume that the core $\omega=\mathcal{X}\cap\mathcal{Y}$ is a contravariantly finite subcategory. We set the following notation for the rest of this subsection. Recall that a morphism $f:A\rightarrow B$ is called an {\it $\omega$-epic}, if the morphism $\Hom_{\mathcal{A}}(W,A)\rightarrow \Hom_{\mathcal{A}}(W,B)$ is epic for each $W\in\omega$.
\begin{itemize}
\item[(1)]
The class of $\omega$-trivial fibrations is defined as
\[TFib_{\omega}:=\{f\in \Mor(\mathcal{A})\mid f\text{ is epic and }\Ker(f)\in\mathcal{Y}\}.\]
\item[(2)]
The class of $\omega$-cofibrations is defined as
\[Cof_{\omega}:={}^{\boxslash}TFib_{\omega}.\]
\item[(3)]
The class of $\omega$-trivial cofibrations is defined as
\[TCof_{\omega}:=\{f\in \Mor(\mathcal{A})\mid f\text{ is a split monic and }\Cok(f)\in\omega\}.\]
\item[(4)]
The class of $\omega$-fibrations is defined as
\[Fib_{\omega}:=\{f\in \Mor(\mathcal{A})\mid f \text{ is } \omega\text{-epic}\}.\]
\item[(5)]
The class of $\omega$-weak equivalences is defined as
\[We_{\omega}:=\{f\in \Mor(\mathcal{A})\mid f=pi\text{ for some }p\in TFib_{\omega} \text{ and some } i\in TCof_{\omega}\}.\]
\end{itemize}

Our aim is to prove that $(Cof_{\omega},We_{\omega},Fib_{\omega})$ is a weak projective model structure.
We will see in Proposition \ref{3.11} that $TCof_{\omega}=Cof_{\omega}\cap We_{\omega}$ and $TFib_{\omega}=Fib_{\omega}\cap We_{\omega}$. To prove these equalities, we need the following result, which guaranties the lifting axiom.

\begin{proposition}\label{3.8}
\begin{itemize}
\item[(1)]
$\omega$-cofibrations have left lifting property with respect to $\omega$-trivial fibratios.
\item[(2)]
$\omega$-trivial cofibrations have left lifting property with respect to $\omega$-fibrations.
\end{itemize}
\begin{proof}
$(1)$ is obvious by the definition of $\omega$-cofibrations.
For $(2)$ let $f$ be a $\omega$-trivial cofibration. So we can assume without lose of generality that $f$ is the morphism $\begin{pmatrix}
1\\
0
\end{pmatrix}:A\rightarrow A\oplus W$, with $W\in\omega$. Now, since $\omega$-fibrations are $\omega$-epic, it is clear that $f$ has left lifting property with respect to $\omega$-fibrations.
\end{proof}
\end{proposition}

The following technical lemma is useful in the sequel. It is well-known also for exact categories. For a direct proof for abelian categories see \cite[Lemma 2.8]{E25}.
\begin{lemma}\label{3.9}
 Let $X\overset{f}{\rightarrow}Y\overset{g}{\rightarrow}Z$ be a pair of composable morphisms in an abelian category. Then we have the following exact sequence.
\begin{center}
 \begin{tikzpicture}
 \node (X1) at (-5,2) {$0$};
 \node (X2) at (-3,2) {$\Ker(f)$};
 \node (X3) at (0,2) {$\Ker(gf)$};
 \node (X4) at (3,2) {$\Ker(g)$};
 \node (X5) at (-3,0) {$\Cok(f)$};
 \node (X6) at (0,0) {$\Cok(gf)$};
 \node (X7) at (3,0) {$\Cok(g)$};
 \node (X8) at (5,0) {$0$};
 \draw [->,thick] (X1) -- (X2) node [midway,above] {};
 \draw [->,thick] (X2) -- (X3) node [midway,above] {};
 \draw [->,thick] (X3) -- (X4) node [midway,above] {};
 \draw [->,thick] (X5) -- (X6) node [midway,above] {};
 \draw [->,thick] (X6) -- (X7) node [midway,left] {};
 \draw [->,thick] (X7) -- (X8) node [midway,left] {};
 \draw [->,thick] (X4) to [out=0,in=180] (X5) node [midway,above] {};
 \end{tikzpicture}
\end{center}
\end{lemma}
\begin{proposition}\label{3.10}
The classes $Cof_{\omega}$, $TCof_{\omega}$, $Fib_{\omega}$ and $TFib_{\omega}$ are all closed under retract and composition, and contain all isomorphisms.
\begin{proof}
Because $Cof_{\omega}={}^{\boxslash}TFib_{\omega}$, it is clear that it is closed under retract and composition and contains all isomorphisms.
We prove these requirements for $TFib_{\omega}$ too, and leave the others to the reader.
First it is obvious that isomorphisms are in $TFib_{\omega}$. Also, because a retract of an epic $f$ is again epic, and its kernel is a summand of $\Ker(f)$, $TFib_{\omega}$ is closed under retract. Now assume that $f:A\rightarrow B$ and $g:B\rightarrow C$ belongs to $TFib_{\omega}$. Clearly $gf$ is epic and by Lemma \ref{3.9} we have the exact sequence
\[0\rightarrow \Ker(f)\rightarrow \Ker(gf)\rightarrow\Ker(g)\rightarrow 0.\]
Then, because $\mathcal{Y}$ is closed under extension, we have that $gf\in TFib_{\omega}$.
\end{proof}
\end{proposition}

\begin{proposition}\label{3.11}
Keeping the above notations,
\begin{itemize}
\item[(1)]
$TCof_{\omega}=Cof_{\omega}\cap We_{\omega}$.
\item[(2)]
$TFib_{\omega}=Fib_{\omega}\cap We_{\omega}$.
\end{itemize}
\begin{proof}
$TCof_{\omega}\subseteq Cof_{\omega}$ is an easy consequence of Lifting-Extension Lemma and $TCof_{\omega}\subseteq We_{\omega}$ is clear. So $TCof_{\omega}\subseteq Cof_{\omega}\cap We_{\omega}$.
Now Let $f\in Cof_{\omega}\cap We_{\omega}$.
Since $f$ is an $\omega$-weak equivalence, by definition there exists a factorization $f=pi$ where $i$ is an $\omega$-trivial cofibration and $p$ is an $\omega$-trivial fibration. Then, the commutative square
\[\begin{tikzcd}
	X & {X'} \\
	Y & Y
	\arrow["i", from=1-1, to=1-2]
	\arrow["f"', from=1-1, to=2-1]
	\arrow["p", from=1-2, to=2-2]
	\arrow["\lambda"{description}, dashed, from=2-1, to=1-2]
	\arrow["id"', from=2-1, to=2-2]
\end{tikzcd}\]
admits a lift $\lambda$, as suggested by the dashed arrow. 
Then by the commutative diagram
\[\begin{tikzcd}
	X & X & X \\
	Y & {X'} & Y
	\arrow["id", from=1-1, to=1-2]
	\arrow["f"', from=1-1, to=2-1]
	\arrow["id", from=1-2, to=1-3]
	\arrow["i"', from=1-2, to=2-2]
	\arrow["f", from=1-3, to=2-3]
	\arrow["\lambda"', from=2-1, to=2-2]
	\arrow["p"', from=2-2, to=2-3]
\end{tikzcd}\]
$f$ is a retract of $i$. Thus by Proposition \ref{3.10} $f$ is $\omega$-trivial cofibration. This proves $(1)$.

For proving $(2)$, $TFib_{\omega}\subseteq We_{\omega}$ is trivial and $TFib_{\omega}\subseteq Fib_{\omega}$ is valid because for any $\omega$-trivial fibration $f:B\rightarrow C$, we have the short exact sequence 
\[0\rightarrow Y\rightarrow B\overset{f}{\rightarrow}C\rightarrow 0,\]
with $Y\in\mathcal{Y}$. Then for every $W\in\omega$, by applying $\Hom_{\mathcal{A}}(W,-)$ to this short exact sequence we obtain that $f$ is $\omega$-epic.
The reverse inclusion $Fib_{\omega}\cap We_{\omega}\subseteq TFib_{\omega}$ is similar to that of $(1)$, and is left to the reader.
\end{proof}
\end{proposition}	
In the following proposition we prove the factorization axiom for $(Cof_{\omega},We_{\omega},Fib_{\omega})$.

\begin{proposition}\label{3.12}
\begin{itemize}
\item[(1)]
Any morphism in $\mathcal{A}$ may be factored as composition of an $\omega$-trivial cofibration followed by an $\omega$-fibration.
\item[(2)]
Any morphism may in $\mathcal{A}$ be factored as composition of an $\omega$-cofibration followed by an $\omega$-trivial fibration.
\end{itemize}
\begin{proof}
Let $f:A\rightarrow B$ be an arbitrary morphism in $\mathcal{A}$. Because $\omega$ is contravariantly finite, we can choose a right $\omega$-approximation $t:W\rightarrow B$ for $B$. Then we can factorize $f$ as in the commutative diagram
\[\begin{tikzcd}
	A && B \\
	& {A\oplus W}
	\arrow["f", from=1-1, to=1-3]
	\arrow["v"', from=1-1, to=2-2]
	\arrow["h"', from=2-2, to=1-3]
\end{tikzcd}\]
where $v=\begin{pmatrix}
1\\0
\end{pmatrix}$ and $h=\begin{pmatrix}
f&t
\end{pmatrix}$. Clearly $v$ is $\omega$-trivial cofibration and $h$ is a $\omega$-fibration. This proves $(1)$.

For proving $(2)$, by definition of $lw$-cotorsion pair, we have an epimorphism $g_B:X_B\rightarrow B$, with $X_B\in\mathcal{X}$. This gives us the solid part of the following commutative diagram with exact first row, where $j=\begin{pmatrix}
1\\0
\end{pmatrix}$ and $r=\begin{pmatrix}
f&g_B
\end{pmatrix}$.
\[\begin{tikzcd}
	&& A &&& \\
	0 & K & {A\oplus X_B} & {} & B & 0 \\
	0 & {Y^K} & E && B & 0
	\arrow["j"', from=1-3, to=2-3]
	\arrow["f", from=1-3, to=2-5]
	\arrow[from=2-1, to=2-2]
	\arrow[from=2-2, to=2-3]
	\arrow["g"', dashed, from=2-2, to=3-2]
	\arrow["r"', from=2-3, to=2-5]
	\arrow["h"', dashed, from=2-3, to=3-3]
	\arrow[from=2-5, to=2-6]
	\arrow["id", dashed, from=2-5, to=3-5]
	\arrow[dashed, from=3-1, to=3-2]
	\arrow[dashed, from=3-2, to=3-3]
	\arrow["p", dashed, from=3-3, to=3-5]
	\arrow[dashed, from=3-5, to=3-6]
\end{tikzcd}\]
By the definition of $hh$-$lw$-cotorsion pair, we have a $\omega$-cofibration $g:K\rightarrow Y^K$, with $Y^K\in\mathcal{Y}$.
Taking push out, we obtain the dashed part of the diagram where the bottom row is also exact.
By the proof of Lemma \ref{2.7}, being a push out of $g$, $h$ is also a $\omega$-cofibration, and so by Proposition \ref{3.10} the composition $hj$ is also an $\omega$-cofibration. Then we have desired factorization $f=p\circ (hj)$.
\end{proof}
\end{proposition}

To prove that $(Cof_{\omega},We_{\omega},Fib_{\omega})$ is a model structure it remains to prove that $We_{\omega}$ is closed under retract, and it satisfies $2$ out of $3$ axiom. To do this we need the following characterization of morphisms in $We_{\omega}$ from \cite{CLZ}.

\begin{lemma}\label{3.13}
The following are equivalent for a morphism $f:A\rightarrow B$.
\begin{itemize}
\item[(1)]
$f\in We_{\omega}$.
\item[(2)]
For any right $\omega$-approximation $t:W\rightarrow B$, $\Ker\big(\begin{pmatrix}f&t\end{pmatrix}:A\oplus W\rightarrow B\big)\in \mathcal{Y}$. Equivalently, we have a factorization of $f$ as composition of an $\omega$-trivial cofibration followed by a $\omega$-trivial fibration, as in the diagram
\[\begin{tikzcd}
	A && B \\
	& {A\oplus W}
	\arrow["f", from=1-1, to=1-3]
	\arrow["e"', from=1-1, to=2-2]
	\arrow["d"', from=2-2, to=1-3]
\end{tikzcd}\]
where $e=\begin{pmatrix}
    1\\
    0
\end{pmatrix}$ and $d=\begin{pmatrix}
    f&t
\end{pmatrix}$.
\end{itemize}
\begin{proof}
   $(1)\Rightarrow (2)$ is proved in \cite[Lemma 3.6]{CLZ}, and $(2)\Rightarrow (1)$ follows from the definition of $We_{\omega}$.
\end{proof}
\end{lemma}

Now we can prove that $\omega$-weak equivalences are closed under retract, and so together with Proposition \ref{3.10} this complete the proof of retract axiom.

\begin{proposition}\label{3.14}
The class $We_{\omega}$ of $\omega$-weak equivalences is closed under retracts.
\begin{proof}
Let $g$ be a retract of $f$. Without lose of generality we may assume that $f$ is of the form $\begin{pmatrix}
g&0\\
0&g'
\end{pmatrix}:X\oplus X'\rightarrow Y\oplus Y'$.
Choose right $\omega$-approximations $t:W\rightarrow Y$ and $t':W'\rightarrow Y'$ for $Y$ and $Y'$ respectively. Thus,
$\begin{pmatrix}
t&0\\
0&t'
\end{pmatrix}:W\oplus W'\rightarrow Y\oplus Y'$ is a right $\omega$-approximation for $Y\oplus Y'$. Consider the factorization
\[\begin{tikzcd}
	{X\oplus X'} && {Y\oplus Y'} & {} \\
	& {X\oplus W\oplus X'\oplus W'} & {}
	\arrow["f", from=1-1, to=1-3]
	\arrow["\alpha"', from=1-1, to=2-2]
	\arrow["\beta"', from=2-2, to=1-3]
\end{tikzcd}\]
for $f$ where
$\alpha=\begin{pmatrix}
1&0\\0&0\\0&1\\0&0
\end{pmatrix}$ and $\beta=\begin{pmatrix}
g&t&0&0\\0&0&g'&t'
\end{pmatrix}$.
Clearly $\alpha\in TCof_{\omega}$ and by Lemma \ref{3.13} $\beta\in TFib_{\omega}$. Because $\begin{pmatrix}
g&t
\end{pmatrix}$ is a retract of $\beta$, by Proposition \ref{3.10} it is in $TFib_{\omega}$. This finishes the proof.
\end{proof}
\end{proposition}

Now we want to prove $2$ out of $3$ axiom for $\omega$-weak equivalences. It should be mentioned that, most of the proofs are inspired from \cite{CLZ}.

\begin{proposition}\label{3.15}
The class $We_{\omega}$ of $\omega$-weak equivalences is closed under composition.
\begin{proof}
The proof of \cite[Lemma 3.5]{CLZ} carries over.
\end{proof}
\end{proposition}

\begin{proposition}\label{3.16}
Let $f$ and $g$ be morphisms in $\mathcal{A}$ such that $gf$ is defined.
If $f$ and $gf$ are $\omega$-weak equivalences, then so is $g$.
\begin{proof}
The proof of \cite[Lemma 3.8]{CLZ} carries over.
\end{proof}
\end{proposition}

The proof of the remaining case $g, gf\in We_{\omega}\Rightarrow f\in We_{\omega}$ is slightly deferent from \cite{CLZ} because here $\omega$-cofibrations are not necessarily monic. But the proofs in \cite{CLZ} can be modified to our setting.

\begin{lemma}\label{3.17}
Let $f\in We_{\omega}$. If we factorize $f$ as $f=pi$ where $i\in Cof_{\omega}$ and $p\in TFib_{\omega}$ (we can always do this by Proposition \ref{3.12}), then $i\in TCof_{\omega}$.
\begin{proof}
Because $f\in We_{\omega}$, by definition, we have a factorization of $f$ as
\[\begin{tikzcd}
	A & {A\oplus W} & B
	\arrow["\begin{array}{c} \begin{pmatrix}1\\0\end{pmatrix} \end{array}", from=1-1, to=1-2]
	\arrow["{(f,t)}", from=1-2, to=1-3]
\end{tikzcd}\]
where $W\in \omega$ and $\Ker(f,t)\in\mathcal{Y}$. So we have the following commutative square.
\[\begin{tikzcd}
	A & {A\oplus W} \\
	C & B
	\arrow["\begin{array}{c} \begin{pmatrix} 1\\ 0 \end{pmatrix} \end{array}", from=1-1, to=1-2]
	\arrow["i"', from=1-1, to=2-1]
	\arrow["{(f,t)}", from=1-2, to=2-2]
	\arrow["\lambda", dashed, from=2-1, to=1-2]
	\arrow["p", from=2-1, to=2-2]
\end{tikzcd}\]
Then by the Lifting axiom, which has been proved in Proposition \ref{3.8}, the lift $\lambda=(\lambda_1,\lambda_2)$ exists.
Then $\lambda_1 i=id$, so $i$ is an split monic. Without lose of generality we can assume that $i=\begin{pmatrix} 1\\ 0 \end{pmatrix}:A\rightarrow A\oplus M$, and then $p=(p_1,p_2)=(f,p_2)$ by the commutativity of the square.

It remains to prove that $M\in \omega=\mathcal{X}\cap\mathcal{Y}$.
Let $0\rightarrow Y\rightarrow E\overset{g}{\rightarrow} M\rightarrow 0$ be a short exact sequence with $Y\in\mathcal{Y}$. By definition $g\in TFib_{\omega}$. Because $i$ has left lifting property with respect to $g$, we can easily see that $g$ is a split epic. So $M\in{}^{\perp_1}\mathcal{Y}=\mathcal{X}$.

Now consider the following factorization for $p=(f,p_2)$.

\begin{center} 
 \begin{tikzpicture}
 \node (X1) at (-4,0) {$A\oplus M$};
 \node (X2) at (0,0) {$A\oplus W\oplus M$};
 \node (X3) at (4,0) {$B$};
 \draw [->,thick] (X1) -- (X2) node [midway,above] {$\begin{pmatrix}
1&0\\0&0\\0&1
\end{pmatrix}$};
 \draw [->,thick] (X2) -- (X3) node [midway,above] {$(f,t,p_2)$};
 \end{tikzpicture}
 \end{center}

By Lemma \ref{3.9} we have the following exact sequence.
\[0\rightarrow\Ker(f,p_2)\rightarrow \Ker(f,t,p_2)\rightarrow \Cok(\begin{pmatrix}
1&0\\0&0\\0&1
\end{pmatrix})\rightarrow 0\]
Since $\Ker(f,p_2)=\Ker(p)\in\mathcal{Y}$, $\Cok\begin{pmatrix}
1&0\\0&0\\0&1
\end{pmatrix}=W\in\mathcal{Y}$, and $\mathcal{Y}$ is closed under extensions, we have that $\Ker(f,t,p_2)\in\mathcal{Y}$.

On the other hand, the short exact sequence
\begin{center} 
 \begin{tikzpicture}
\node (X0) at (-6,0) {$0$};
 \node (X1) at (-4,0) {$\Ker(f,t,p_2)$};
 \node (X2) at (0,0) {$A\oplus W\oplus M$};
 \node (X3) at (4,0) {$B$};
\node (X4) at (6,0) {$0$};
 \draw [->,thick] (X1) -- (X2) node [midway,above] {$\begin{pmatrix}
a\\b\\c
\end{pmatrix}$};
 \draw [->,thick] (X2) -- (X3) node [midway,above] {$(f,t,p_2)$};
\draw [->,thick] (X0) -- (X1) node [midway,above] {};
\draw [->,thick] (X3) -- (X4) node [midway,above] {};
 \end{tikzpicture}
 \end{center}
gives us the following pull back diagram
\[\begin{tikzcd}
	0 & {\Ker(f,t)} & {\Ker(f,t,p_2)} & M & 0 \\
	0 & {\Ker(f,t)} & {A\oplus W} & B & 0
	\arrow[from=1-1, to=1-2]
	\arrow[from=1-2, to=1-3]
	\arrow["id"', from=1-2, to=2-2]
	\arrow["-c", from=1-3, to=1-4]
	\arrow["\bar{a}", from=1-3, to=2-3]
	\arrow[from=1-4, to=1-5]
	\arrow["p_2", from=1-4, to=2-4]
	\arrow[from=2-1, to=2-2]
	\arrow[from=2-2, to=2-3]
	\arrow["{(f,t)}", from=2-3, to=2-4]
	\arrow[from=2-4, to=2-5]
\end{tikzcd}\]
where $\bar{a}:=\begin{pmatrix}
    a\\b
\end{pmatrix}$. By the hereditary assumption on the $lw$-cotorsion pair $(\mathcal{X},\mathcal{Y})$ and exactness of the first row we see that $M\in\mathcal{Y}$.
\end{proof}
\end{lemma}

\begin{proposition}\label{3.18}
Suppose that $gf, g\in We_{\omega}$. Then $f\in\omega$.
\begin{proof}
Since $g\in We_{\omega}$, we have the solid part of the following commutative diagram where $W\in\omega$,
$j=\begin{pmatrix}1\\ 0\end{pmatrix}$, and $h\in TFib_{\omega}$.
\[\begin{tikzcd}
	A && B && C \\
	& P && {B\oplus W}
	\arrow["f", from=1-1, to=1-3]
	\arrow["i"', dashed, from=1-1, to=2-2]
	\arrow["g", from=1-3, to=1-5]
	\arrow["j", from=1-3, to=2-4]
	\arrow["p"', dashed, from=2-2, to=2-4]
	\arrow["h"', from=2-4, to=1-5]
\end{tikzcd}\]
Then by the factorization axiom, which is proved in Proposition \ref{3.12}, we have the factorization $jf=pi$ with $p\in TFib_{\omega}$ and $i\in Cof_{\omega}$, as suggested by the dashed arrows.
Because $gf=(hp)i\in We_{\omega}$, $hp\in TFib_{\omega}$ and $i\in Cof_{\omega}$, by Lemma \ref{3.17} we have $i\in TCof_{\omega}$. So $jf=\begin{pmatrix}f\\0\end{pmatrix}\in We_{\omega}$. Since $f$ is clearly a retract of $jf$, by Proposition \ref{3.14} $f\in We_{\omega}$.
\end{proof}
\end{proposition}
Let's summarize what we have proven in Subsection 3.2.
\begin{proposition}\label{3.19}
    Let $(\mathcal{X},\mathcal{Y})$ be a $hh$-$lw$-cotorsion pair in $\mathcal{A}$ with a contravariantly finite core $\omega=\mathcal{X}\cap\mathcal{Y}$.
    Then $(Cof_{\omega},We_{\omega},Fib_{\omega})$ is a weak projective model structure on $\mathcal{A}$.
    \begin{proof}
    The lifting axiom is proved in Proposition \ref{3.8}. The retract axiom is proved in Propositions \ref{3.10} and \ref{3.14}.
    The factorization axiom is proved in Proposition \ref{3.12}. And $2$ out of $3$ axiom is proved in Propositions \ref{3.15}, \ref{3.16} and \ref{3.18}.
    So $(Cof_{\omega},We_{\omega},Fib_{\omega})$ is a model structure.

    By definition, $TFib_{\omega}$ coincides with the class of all epics with kernel in $\mathcal{Y}$ (=the class of trivially fibrant objects). Also, for every object $A\in\mathcal{A}$, one can easily see that $A\rightarrow 0$ has right lifting property with respect to all morphisms in $TCof_{\omega}$. Thus all objects are fibrant, which means that $(Cof_{\omega},We_{\omega},Fib_{\omega})$ is a weak projective model structure.
    \end{proof}
\end{proposition}

Now we can prove Theorem \ref{1.2}, which is an analogous of Beligiannis-Reiten correspondence for $lw$-cotorsion pairs.

\begin{theorem}\label{3.20}
Let $\mathcal{A}$ be an abelian category. There exists a bijective correspondence between weak projective model structures on $\mathcal{A}$ and $hh$-$lw$-cotorsion pairs in $\mathcal{A}$ with contravariantly finite core. This correspondence acts by
\begin{align*}
\Phi:\mathcal{M}=(Cof,We,Fib)\longmapsto (\text{Cofibrant objects},\text{Trivially fibrant objects})
\end{align*}
and the inverse is given by
\begin{align*}
\Psi:\mathcal{C}=(\mathcal{X},\mathcal{Y})\longmapsto (\omega\text{-Cofibrations},\omega\text{-Weak equivalences}, \omega\text{-Fibrations})
\end{align*}
where $\omega=\mathcal{X}\cap\mathcal{Y}$.
\begin{proof}
By Proposition \ref{3.5} $\Phi$ is well-defined, and by Proposition \ref{3.19} $\Psi$ is well-defined.
Now we prove that they are inverse of each other.

In order to prove that $\Psi\Phi$ is identity, we must show that for any weak projective model structure $(Cof,We,Fib)$, we have
\[(Cof,We,Fib)=(\omega\text{-Cofibrations},\omega\text{-Weak equivalences}, \omega\text{-Fibrations}),\]
where $\omega$ is the class of trivially cofibrant objects.
Clearly $TFib=TFib_{\omega}$.
Because any model structure is uniquely determined by two classes of trivial fibrations and trivial cofibrations, it is enough to prove that
\[TCof=TCof_{\omega}:=\{\text{split monics with kernel in }\omega\}.\]
Every split monic with cokernel in $\omega$ is (isomorphic to) a morphism of the form $i:=\begin{pmatrix}
    1\\0
\end{pmatrix}:A\rightarrow A\oplus W$ with $W\in\omega$. Because $1:A\rightarrow A$ and $0\rightarrow W$ are both in $TCof$, so is $i$.
This proves the inclusion $\supseteq$. Since cokernel of any trivial cofibration is trivially cofibrant, to prove the inclusion $\subseteq$ it is enough to prove that any trivial cofibration in $(Cof,We,Fib)$ is a split monic. For a trivial cofibration $X\rightarrow Y$, because all objects are fibrant by definition, the following square admits a lift as suggested by the dashed arrow, and this proves the claim.
\[\begin{tikzcd}
	X & X \\
	Y & 0
	\arrow["\id", from=1-1, to=1-2]
	\arrow[from=1-1, to=2-1]
	\arrow[from=1-2, to=2-2]
	\arrow[dashed, from=2-1, to=1-2]
	\arrow[from=2-1, to=2-2]
\end{tikzcd}\]
Now let prove that $\Phi\Psi$ is the identity map. Let $(\mathcal{X},\mathcal{Y})$ be a $hh$-$lw$-cotorsion pair with the contravariantly finite core $\omega$. We need to prove that for the model structure
\[(\omega\text{-Cofibrations},\omega\text{-Weak equivalences}, \omega\text{-Fibrations}),\]
the class of cofibrant objects is equival to $\mathcal{X}$ and the class of trivially fibrant objects is equal to $\mathcal{Y}$. The latter is trivial by the definition of $\omega$-trivial fibrations.
For the former, we prove that for an object $A\in\mathcal{A}$, $A$ is a cofibrant object if and only if $A\in {}^{\perp_1}\mathcal{Y}=\mathcal{X}$.
If $A$ is a cofibrant object, then any short exact sequence $0\rightarrow Y\rightarrow E\rightarrow A\rightarrow 0$ splits because $E\rightarrow A$ is an $\omega$-trivial fibration. Conversely, if $A\in {}^{\perp_1}\mathcal{Y}$, then $0\rightarrow A$ has left lifting property with respect to all epics with kernel in $\mathcal{Y}$, which means that $0\rightarrow A$ is a $\omega$-cofibration.
\end{proof}
\end{theorem}
The following result recovers Beligiannis-Reiten correspondence, and shows that the correspondence of Theorem \ref{3.20} is a generalization of their correspondence.

\begin{corollary}\label{3.21}
Let $\mathcal{A}$ be an abelian category. There exists a bijective correspondence between projective model structures on $\mathcal{A}$ and hereditary cotorsion pairs in $\mathcal{A}$ with contravariantly finite core. This correspondence acts by
\begin{align*}
\Phi|:\mathcal{M}=(Cof,We,Fib)\longmapsto (\text{Cofibrant objects},\text{Trivially fibrant objects})
\end{align*}
and the inverse is given by
\begin{align*}
\Psi|:\mathcal{C}=(\mathcal{X},\mathcal{Y})\longmapsto (\omega\text{-Cofibrations},\omega\text{-Weak equivalences}, \omega\text{-Fibrations})
\end{align*}
where $\omega=\mathcal{X}\cap\mathcal{Y}$.
\begin{proof}
Because projective model structures are a subclass of weak projective model structures, and hereditary cotorsion pairs are a subclass of $hh$-$lw$-cotorsion pairs, by Theorem \ref{3.20} we only need to show that $\Phi|$ and $\Psi|$ are well-defined.

Let $(Cof,We,Fib)$ be a projective model structure. because cofibrations are monic, the $lw$-cotorsion pair $(\mathpzc{C},\mathpzc{TF})$ is indeed a cotorsion pair. This proves that $\Phi|$ is well-defined.

Conversley, let $(\mathcal{X},\mathcal{Y})$ be a hereditary cotorsion pair with 
contravariantly finite core $\omega$. Then by Theorem \ref{3.20} $(\omega\text{-Cof},\omega\text{-We},\omega\text{-Fib})$ is a weak projective model structure, and if we show that cofibrations are all monic, it would be a projective model structure. So let $f:A\rightarrow B$ be a $\omega$-cofibration. By the definition of cotorsion pair we 
have a monic $A\rightarrow Y$ with $Y\in\mathcal{Y}$. Then in the following commutative diagram, the 
existence of the lift $\lambda$ shows that $f$ is also a monic.
\[\begin{tikzcd}
	A & Y \\
	B & 0
	\arrow[tail, from=1-1, to=1-2]
	\arrow["f"', from=1-1, to=2-1]
	\arrow[from=1-2, to=2-2]
	\arrow["\lambda"{description}, dashed, from=2-1, to=1-2]
	\arrow[from=2-1, to=2-2]
\end{tikzcd}\]
Thus, $\Psi|$ is also well-defined.
\end{proof}
\end{corollary}

\section{A homotopical interpretation of support $\tau$-tilting modules}
Let $\Lambda$ be an Artin algebra and $\modd\text{-}\Lambda$ be the abelian category of finitely generated left $\Lambda$-modules. As always, $\tau$ is the Auslander-Reiten translation.
Beligiannis and Reiten proved that (co)tilting modules in $\modd\text{-}\Lambda$ are in bijection with a certain class of model structures on $\modd\text{-}\Lambda$ \cite[VIII, Theorem 5.8]{BR}. In this section we prove a similar result for support $\tau$-tilting modules.
\begin{definition}$($\cite{AIR}$)$
Let $T\in\modd\text{-}\Lambda$.
\begin{itemize}
\item
$T$ is called {\it $\tau$-rigid}, if $\Hom_{\Lambda}(T,\tau T)=0$.
\item
$T$ is called a {\it $\tau$-tilting module}, if $T$ is $\tau$-rigid and the number of non-isomorphic summands of $T$ is equal to that of $\Lambda$ itself.
\item
$T$ is called a {\it support $\tau$-tilting module} if $T$ is $\tau$-tilting module over its support (i.e. there is an idempotent $e\in \Lambda$ such that $T$ is
a $\tau$-tilting module in $\modd\text{-}\Lambda/<e>$. 
\end{itemize}
\end{definition}

We need the following two results from $\tau$-tilting theory. Recall that for a given module $T\in\modd\text{-}\Lambda$, $\add T$ is the subcategory of all modules in $\modd\text{-}\Lambda$ which are direct summand of a finite direct sum of copies of $T$. And $\Fac T$ is the subcategory of all modules in $\modd\text{-}\Lambda$ which are a factor of a module in $\add T$.

\begin{lemma}$($\cite[Proposition 2.14]{J}\label{4.2}
A $\tau$-rigid $\Lambda$-module $T$ is a support $\tau$-tilting module if and only if there exists an exact sequence
\begin{equation}\label{eq4.1}
\Lambda\overset{f}{\longrightarrow}T_0\overset{g}{\longrightarrow} T_1\rightarrow 0
\end{equation}
with $T^0,T^1\in \add T$ and $f$ a left $(\add T)$-approximation for $\Lambda$. In this case, since $\Lambda$ is projective, $f$ is a left $(\Fac T)$-approximation too.
\end{lemma}

\begin{proposition}\label{4.3}
Let $T\in \modd\text{-}\Lambda$ be a support $\tau$-tilting module.
\begin{itemize}
\item[(1)]
$\Fac T$ is a torsion class, i.e. it is closed under extensions and factor modules.
\item[(2)]
$\Fac T$ is functorially finite.
\end{itemize}
\begin{proof}
The proofs can be found in \cite{AIR}.
Because the proof of $(2)$ is crucial for the main result of this section, we sketch the proof.
Since all torsion classes are contravariantly finite, we need to argue why $\Fac T$ is covariantly finite?
For a given module $M\in\modd\text{-}\Lambda$, choose an epic $\pi:\Lambda^n\rightarrow M$.
Then by Lemma \ref{4.2} we have the solid part of the following diagram where $f^n:=\begin{pmatrix}
f&0&\cdots&0\\
0&f&\cdots&0\\
\ldots&\ldots&\ldots&\ldots\\
0&0&\cdots&f
\end{pmatrix}$ and the same for $g^n$.
\[\begin{tikzcd}
	{\Lambda^n} & {T_0^n} & {T_1^n} & 0 \\
	M & E & {T_1^n} & 0 \\
	0 & 0
	\arrow["{f^n}", from=1-1, to=1-2]
	\arrow["\pi"', from=1-1, to=2-1]
	\arrow["{g^n}", from=1-2, to=1-3]
	\arrow["\phi"', dashed, from=1-2, to=2-2]
	\arrow[from=1-3, to=1-4]
	\arrow["id"', dashed, from=1-3, to=2-3]
	\arrow["h"', dashed, from=2-1, to=2-2]
	\arrow[from=2-1, to=3-1]
	\arrow[dashed, from=2-2, to=2-3]
	\arrow[dashed, from=2-2, to=3-2]
	\arrow[dashed, from=2-3, to=2-4]
\end{tikzcd}\]
By taking push out along $\pi$, we obtain the dashed arrows, making the whole diagram commutative with exact rows and columns. $f^n$ is a left $(\Fac T)$-approximation because so is $f$. Then, by the universal property of push out $h$ is a left $(\Fac T)$-approximation for $M$.
\end{proof}
\end{proposition}

\begin{proposition}\label{4.4}
Let $T\in \modd\text{-}\Lambda$ be a support $\tau$-tilting module.
\begin{itemize}
\item[(1)]
$({}^{\perp_1}(\Fac T),\Fac T)$ is a $hh$-$lw$-cotorsion pair in $\modd\text{-}\Lambda$, and the core is $\add T$, which is contravariantly finite.
\item[(2)]
Sending $T$ to $({}^{\perp_1}(\Fac T),\Fac T)$ gives a bijective correspondence between support $\tau$-tilting $\Lambda$-modules and and $hh$-$lw$-cotorsion pairs $(\mathcal{X},\mathcal{Y})$ with $\mathcal{Y}$ a torsion pair.
\end{itemize}
\begin{proof}
Most of the requirements has been proved in \cite{BZ}.
The only thing that remains to prove is that the $lw$-cotorsion pair $({}^{\perp_1}(\Fac T),\Fac T)$ is a homotopicaly $lw$-cotorsion pair.

First, let us prove that in the exact sequence \eqref{eq4.1} given in the Lemma \ref{4.2} $f$ has left lifting property with respect to all epics with kernel in $\Fac T$. So, assume that we are given the exact commutative diagram 
\[\begin{tikzcd}
	&& \Lambda & {T_0} & {T_1} & 0 \\
	0 & Y & B & C & 0
	\arrow["f", from=1-3, to=1-4]
	\arrow["a"', from=1-3, to=2-3]
	\arrow["g", from=1-4, to=1-5]
	\arrow["b", from=1-4, to=2-4]
	\arrow[from=1-5, to=1-6]
	\arrow[from=2-1, to=2-2]
	\arrow["r", from=2-2, to=2-3]
	\arrow["s", from=2-3, to=2-4]
	\arrow[from=2-4, to=2-5]
\end{tikzcd}\]
with $Y\in\Fac T$. Since $\Ext_{\Lambda}^1(T,Y)=0$, there exists a morphism $\lambda_1:T_0\rightarrow B$ such that $s\lambda_1=b$.
Then, because $s(a-\lambda_1f)=sa-bf=0$, and $\Lambda$ is a projective module, there exists a morphism $\lambda_2:\Lambda\rightarrow Y$ such that $a-\lambda_1f=r\lambda_2$.
Finally, because $f$ is a left $(\Fac T)$-approximation, there exists a morphism $\lambda_3:T_0\rightarrow Y$ such that
$\lambda_2=\lambda_3f$. Then for $\lambda:=\lambda_1+r\lambda_3$ we can see that
\begin{align*}
\lambda f&=\lambda_1f+r\lambda_3f\\
              &=(a-r\lambda_2)+r\lambda_3f\\
              &=a-r(\lambda_2-\lambda_3f)\\
              &=a, \text{and}\\
s\lambda &=s(\lambda_1+r\lambda_3)\\
              &=s\lambda_1\\
              &=b.
\end{align*}
Thus we have proved that $f$ has left lifting property with respect to all epics with kernel in $\Fac T$.
As we saw in the proof of Proposition \ref{4.3}, for any positive integer $n$, $f^n$ also satisfies this condition, and for any $M\in\modd\text{-}\Lambda$, we have a short exact sequence $M\overset{h}{\longrightarrow}E\rightarrow T'\rightarrow 0$, such that $h$ is a push out of $f^n$ for some positive integer $n$, $E\in \Fac T$ and $T'\in \add T$.
Since $h$ is a push out of $f^n$, by the proof of Lemma \ref{2.7} it has  left lifting property with respect to all epics with kernel in $\Fac T$. This completes the proof.
\end{proof}
\end{proposition}

Now we can prove the main result of this section which gives a homotopical interpretation of support $\tau$-tilting modules.

\begin{theorem}\label{4.5}
Let $\Lambda$ be an Artin algebra. There is a bijective correspondence between:
\begin{itemize}
\item[(1)]
Isoclasses of basic support $\tau$-tilting $\Lambda$-modules.
\item[(2)]
Model structures $(Cof,We,Fib)$ in $\modd\text{-}\Lambda$ satisfying the following conditions.
\begin{itemize}
\item[(i)]
All modules in $\modd\text{-}\Lambda$ are fibrant.
\item[(ii)]
Trivial fibrations coincides with epics with trivially fibrant kernel.
\item[(iii)]
The subcategory of trivially fibrant objects is closed under factor modules.
\end{itemize}
\end{itemize}
\begin{proof}
For any support $\tau$-tilting $\Lambda$-modules $T$, $({}^{\perp_1}(\Fac T),\Fac T)$ is a $hh$-$lw$-cotorsion pair, with the contravariantly finite core $\add T$. Then, by Theorem \ref{3.20} we get the model structure
\[\big((\add T)\text{-}Cof,(\add T)\text{-}We,(\add T)\text{-}Fib\big),\]
which satisfies the requirements in $(2)$.
Conversely, let $(Cof,We,Fib)$ be a model structure satisfying the requirements in $(2)$.
Then the associated $hh$-$lw$-cotorsion pair under the bijection of Theorem \ref{3.20} is $(\mathpzc{C},\mathpzc{TF})$.
By assumption, $\mathpzc{TF}$ is a torsion class. Thus by Proposition \ref{4.4} $(\mathpzc{C},\mathpzc{TF})=({}^{\perp_1}(\Fac T),\Fac T)$ for a unique support $\tau$-tilting module $T$.
\end{proof}
\end{theorem}

\section*{acknowledgments}
Ramin Ebrahimi is supported by Zhejiang Normal University.
Rasool Hafezi is supported by the National Natural Science of China (Grant No. 12571042).
Jiaqun Wei is supported by the National Natural Science Foundation of China (Grant Nos. 12571042, 12271249) and the Natural Science Foundation of Zhejiang Province (Grant No. LZ25A010002).
	

\end{document}